\documentclass[12pt,a4paper,final]{amsart}
\usepackage[cp1250]{inputenc}
\usepackage[T1]{fontenc}

\usepackage{amsfonts}
\usepackage{amsmath}
\usepackage{amssymb}
\usepackage{amsthm}
\usepackage{mathrsfs}
\usepackage{amsopn}
\usepackage{indentfirst}
\usepackage{enumerate}
\usepackage{graphicx}
\usepackage{xcolor}

\newtheoremstyle{df}{8pt}{}{}{0cm}{\bf}{.\ }{0pt}{}
\numberwithin{equation}{section}
\newtheorem{tw}{Theorem}[section]
\newtheorem{lm}[tw]{Lemma}
\newtheorem{stw}[tw]{Proposition}
\newtheorem{wn}[tw]{Corollary}
\newtheorem{uw}[tw]{Remark}
\theoremstyle{df}
\newtheorem{df}[tw]{Definition}
\newtheorem{ex}[tw]{Example}

\newcommand{\R}{\mathbb{R}}

\newcommand{\X}{\mathcal{X}}
\newcommand{\XM}{\X(M)}

\newcommand{\Span}{\operatorname{span}}

\newcommand{\eps}{\varepsilon}

\newcommand{\Cniesk}{{C^\infty}}

\newcommand{\ind}{\operatorname{ind}}
\newcommand{\Rad}{\operatorname{Rad}}

\begin{document}
\baselineskip=17pt
\title{Contact variation of almost contact pseudo-metric structures}
\author[A. Bor\'owka]{Aleksandra Bor\'owka}
\address{Aleksandra Bor\'owka, Institute of Mathematics, Jagiellonian University in Krak\'ow, Poland}
\email{aleksandra.borowka@uj.edu.pl}
\author[Z. Szancer]{Zuzanna Szancer}\address{Zuzanna Szancer, Department of Applied Mathematics, University of Agriculture in Krakow, 253 Balicka St., 30-198 Krakow, Poland}
\email{Zuzanna.Szancer@urk.edu.pl}


\begin{abstract}
For an almost contact pseudo-metric manifold we study a variation of the metric in the spirit of Meli-Ngakeu-Olea \cite{MNO}. Such variation is defined using a vector field on the manifold and the procedure shifts the signature of the metric by increasing the positive index by two. First we derive conditions on the vector field for the obtained manifold to remain almost contact pseudo-metric. Then we study the cases of almost pseudo co-K\"ahler  manifolds. In particular we obtain geometric characterizations of variations for which the fundamental form remains closed and establish a local decomposition into a product. We further investigate the variation in the setting of statistical structures.
\end{abstract}

\maketitle

\section{Introduction}

Almost contact metric manifolds are the subject of intensive interest both from a purely mathematical perspective as well as due to their numerous applications to mathematical physics (see for example \cite{ADDK},\cite{KKMI},\cite{PysC}). In the pseudo-Riemannian context, they were first considered by Takahashi \cite{Tak} and subsequently discussed in several works \cite{CP}, \cite{Perrone}. While initial focus was placed mostly 0on studies of the Lorentzian case with a timelike Reeb vector field (\cite{Calvaruso}, \cite{D}), recent developments in certain areas of physics, such as relativistic magneto-fluids or multi-time string theories, motivate also the study of the pseudo-Riemannian case when the Reeb vector field is spacelike and the timelike directions are contact. In this paper, we study a canonical variation of almost contact pseudo-metric structures in the spirit of the work \cite{MNO}. Using a vector field on the manifold, we deform the metric in a way that changes its signature by increasing the positive index by \(2\). In the context of Lorentzian contact manifolds, analogous construction has been studied in its simplest form where the Reeb vector field is timelike (see \cite{Calvaruso}, where this case is called \(\mathcal{D}\)-homothetic transformation). Therefore, in the present work, we mostly consider the remaining case where the Reeb vector field is spacelike. We first determine the necessary and sufficient conditions on the vector field for the target pseudo-metric to remain almost contact. Next, we investigate the particularly interesting case of almost co-K\"ahler metrics, deriving analogous conditions. Finally, we study the case of almost pseudo-metric statistical structures.
\section{Preliminaries}\label{sec::Preliminaries}
Let us firstly introduce the main objects of this paper.
\begin{df} \begin{enumerate}
    \item 

A $(2n+1)$-dimensional manifold $M$ is equipped with an
\emph{almost contact structure} if there exist on $M$ a tensor
field $\varphi$ of type (1,1), a vector field $\xi$ and a 1-form
$\eta$ which satisfy
$$
\varphi^2(X)=-X+\eta(X)\xi, \ \ \ 
\eta(\xi)=1
$$
for every $X\in TM$.

\item If additionally there is a pseudo-metric
$g$ on $M$ such that
$$
g(\varphi X,\varphi Y)=g(X,Y)-\eps\eta(X)\eta(Y),\quad \eps \in \{-1,1\}
$$
for every $X,Y\in TM$ then $(\varphi,\xi,\eta,g)$ is called an
\emph{almost contact pseudo-metric structure}. 
\item An almost contact pseudo-metric
structure is called  \emph{Sasakian} if
$$
(\widehat{\nabla}_X\varphi)(Y)=g(X,Y)\xi-\eps\eta(Y)X,
$$
where $\widehat{\nabla}$ is the Levi-Civita connection for $g$. 
\item An
almost contact pseudo-metric structure $(\varphi,\xi,\eta,g)$ is called a
\emph{contact pseudo-metric structure} if

$$g(X,\varphi Y)=d\eta(X,Y)$$

for every $X,Y\in TM$.
\item
We say that an almost contact structure $(\varphi,\xi,\eta)$ is
\emph{normal} if
$$
[\varphi,\varphi]+2d\eta\otimes\xi=0,
$$
where $[\varphi,\varphi]$ is the Nijenhuis tensor for $\varphi$.
\end{enumerate}
\end{df}
\begin{df}
Let $(M,\varphi,\xi,\eta,g)$ be an almost contact pseudo-metric manifold. $M$ is called almost co-K\"ahler pseudo-metric manifold if
$$
d\eta=0,\qquad d \Phi=0,
$$
where $\Phi$ is the fundamental form, that is $\Phi (X,Y)=g(X,\varphi Y)$.
\end{df}

\begin{df}
$M$ is called  co-K\"ahler pseudo-metric manifold if it is almost co-K\"ahler and normal.
\end{df}

The following property is well known:
\begin{uw}\label{uw::cok::char}
Let $(M,\varphi,\xi,\eta,g,\eps)$ be an almost contact pseudo-metric manifold
and let $\widehat{\nabla}$ be the Levi-Civita connection of $g$. Then
$(M,\varphi,\xi,\eta,g,\eps)$ is co-K\"ahler if and only if
$$
\widehat{\nabla}\varphi=0.
$$
\end{uw}
\section{Almost contact variation of an almost contact pseudo-metric manifold}
Our first aim is to define \emph{contact variation of a metric along a vector field} in a similar way it has been done for Hermitian case in \cite{MNO}.
Let $(M,\varphi,\xi,\eta,g,\varepsilon)$ be an almost contact pseudo-metric manifold and let $F_1,F_2\in\XM$. We define a symmetric tensor of type $(0,2)$ by the formula:
\begin{equation}\label{eq::Gast}
g^\ast(X,Y)=g(X,Y)+g(X,F_1)g(Y,F_1)+g(X,F_2)g(Y,F_2).
\end{equation}
\begin{lm}\label{lm::gast::general}

The metric $g^\ast$ defined above is non-degenerate if and only if
\begin{equation}\label{eq::GastNonDeg}
(1+g(F_1,F_1))(1+g(F_2,F_2))-g(F_1,F_2)^2\neq 0.
\end{equation}

If $g^\ast$ is non-degenerate then $(M,\varphi,\xi,\eta,g^\ast,\varepsilon^\ast)$ is
an almost contact pseudo-metric manifold if and only if
\begin{enumerate}
  \item $F_1=b_1\xi, F_2=b_2\xi$, $b_1^2+b_2^2=2$, $\varepsilon = -1$, $\varepsilon^\ast=1$ or
  \item $F_1, F_2\in\ker\eta$, $F_2=\pm \varphi F_1$, $\varepsilon = \varepsilon^\ast$.
\end{enumerate}
\end{lm}
\begin{proof}
Let us define
$$
A(Y):=Y+g(Y,F_1)F_1+g(Y,F_2)F_2.
$$ 
Then $g^\ast$ is degenerate if and only if $\ker A\neq 0$.
If $Y\in\XM$, $Y\neq 0$, $Y\in\ker A$ then
$$
Y=\alpha F_1+\beta F_2
$$
and by analyzing two cases based on whether $F_1$ and $F_2$ are linearly dependent, a straightforward computation yields the claim.

\par In order to prove the second part of the lemma first note that we can decompose
$$
F_1=U_1+b_1\xi,\quad F_2=U_2+b_2\xi
$$
where  $U_1,U_2\in\ker\eta$ and $b_1,b_2\in \Cniesk(M)$. The direct check of the almost contact metric condition implies that if we assume that $b_1\neq 0$ we obtain the first condition, while for $b_1=0$ we obtain the second one.
\end{proof}
\begin{uw}
Taking $$
F_1:=\xi,\quad F_2:=\xi
$$ we obtain a particular change of metric
$$
\tilde{g}(X,Y)=g(X,Y)+g(X,\sqrt{2}\xi)g(Y,\sqrt{2}\xi).
$$
This is the \(\mathcal{D}\)-homothetic transformation studied in \cite{Calvaruso} in the case when $g$ is Lorentzian and $\xi$ timelike as a way of passing from Lorentzian to Riemannian case. Therefore in this paper we will focus on the case when $\xi$ is spacelike and we change the index in the contact direction.

\end{uw}

Let us now study what is the index of $g^\ast$ when $F\in\ker\eta$ and $g(F,F)\neq 0$. We define $$
H:=\operatorname{span} \{F,\varphi F\}
.$$ 
Then $H$
is $2$-dimensional subspace of $TM$ and
$
TM=H\oplus H^{\perp_g}
.$ The following proposition implies in particular that in some cases this allows us to pass from $(n-2,2)$ signature to the Riemannian case.
\begin{stw}
\label{lem::indeks:gstar}
Let $(M,\varphi,\xi,\eta,g)$ be an almost contact pseudo-metric manifold and let $F\in\ker\eta$. On $M$ we define a symmetric tensor of type $(0,2)$ by Formula (\ref{eq::Gast}).
If $g^\ast$ is nondegenerate then
\begin{equation}\label{eq::indeks:gstar}
\ind g^\ast=
\begin{cases}
\ind g-2,& \lambda<-1,\\
\ind g,& \lambda>-1,
\end{cases}
\end{equation}
where $\lambda:=g(F,F)$.
Moreover, $g^\ast$ is degenerate if and only if $\lambda=-1$.
\end{stw}
\begin{proof}
First note that since $g^\ast$ is nondegenerate, Lemma \ref{lm::gast::general} implies $\lambda=g(F,F)\neq -1$.

First assume $\lambda\neq 0$ then
$$
\ind g^\ast=\ind g^\ast|_H+\ind g^\ast|_{H^{\perp_g}} = \ind g^\ast|_H+\ind g|_{H^{\perp_g}}
$$
since $g^\ast=g$ on $H^{\perp_g}$.
By straightforward computations we obtain
\begin{align*}
g^\ast (F,\varphi F)&=0,\\
g^\ast(F,F)&=g^\ast (\varphi F,\varphi F)=\lambda (\lambda +1).
\end{align*}
Therefore
$$
\ind g^\ast|_H=
\begin{cases}
0,& \lambda\in(-\infty,-1)\cup(0,+\infty),\\
2,& \lambda\in(-1,0)
\end{cases}
$$
and
$$
\ind g|_H=
\begin{cases}
0,& (0,+\infty),\\
2,& \lambda\in(-\infty,0).
\end{cases}
$$
In consequence
$$
\ind g^\ast|_H=
\begin{cases}
\ind g|_H,& \lambda\in (-1,0)\cup(0,+\infty),\\
\ind g|_H-2,& \lambda\in(-\infty,-1)
\end{cases}
$$
what proves (\ref{eq::indeks:gstar}) in case $\lambda\neq 0$.
\par Now assume that $\lambda = 0$.  If $F=0$, then $g^\ast=g$ and $\ind g^\ast=\ind g$ and (\ref{eq::indeks:gstar}) holds trivially.
If $F\neq 0$ then $\dim\ker\eta\geq 4$ and due to nondegeneracy of $g$ there exists $X\in\ker\eta$ such that
$$
g(X,F)=1 \quad\text{and}\quad g(X,\varphi F)=0.
$$
Let
$$
V:= \operatorname{span}\{F,X,\varphi F,\varphi X \}
$$
We have
$$
g|_V\sim
\left[\begin{matrix}
  0 & 1 & 0 & 0 \\
  1 & \mu & 0 & 0 \\
  0 & 0 & 0 & 1 \\
  0 & 0 & 1 & \mu
\end{matrix}\right],\quad
g^\ast|_V\sim
\left[\begin{matrix}
  0 & 1 & 0 & 0 \\
  1 & \mu+1 & 0 & 0 \\
  0 & 0 & 0 & 1 \\
  0 & 0 & 1 & \mu+1
\end{matrix}\right]
$$
where $\mu = g(X,X)$.
Now
$$
\ind g^\ast = \ind g^\ast|_V+\ind g^\ast|_{V^{\perp_g}} = \ind g
$$
since
$$
\ind g^\ast|_V =2=  \ind g|_V
$$
and $g^\ast = g$ on $V^{\perp_g}$ what completes the proof in case $\lambda = 0$.
\end{proof}
A particular case from our perspective is when $g$ is an almost co-K\"ahler pseudo-metric. The next theorem gives necessary and sufficient conditions that $g^{\ast}$ is also pseudo almost co-K\"ahler.
\begin{tw}\label{lm::acokequiv}
Let $(M,\varphi,\xi,\eta,g)$ be an almost co-K\"ahler pseudo-metric manifold and $F\in \ker\eta$. Let $g(F,F)\neq -1$ and $g(F,F)\neq 0$. Then the following conditions are equivalent:
\begin{enumerate}
\item $(M,\varphi,\xi,\eta,g^{\ast})$ \text{is an almost co-K\"ahler pseudo-metric manifold},
\item $d(g(F,\cdot)\wedge g(\varphi F,\cdot))=0$,
\item $H^{\perp_g}$ is an integrable distribution and for all  $X\in H^{\perp_g}$ $$X(g(F,F))=g(X,\nabla_FF+\nabla_{\varphi F}\varphi F).$$ 
\end{enumerate}
\end{tw}
\begin{proof}
First note that
\begin{align*}
\Phi ^{\ast}(X,Y):=g^{\ast}(X,\varphi Y)=\Phi (X,Y)-(g(F,\cdot)\wedge g(\varphi F,\cdot))(X,Y)
\end{align*}
Since $d\eta =0$ and $d \Phi =0$ we immediately get $d\Phi^{\ast}=0$ if and only if $d(g(F,\cdot)\wedge g(\varphi F,\cdot))=0$. Thus (1) and (2) are equivalent.
\par 
It is easy to see that $H^{\perp_g}=\ker\alpha \cap\ker\beta$, where $\alpha:=g(F,\cdot)$ and $\beta:=g(\varphi F,\cdot).$
Since $d(\alpha\wedge\beta)=0$ from the Frobenius theorem we get that $H^{\perp_g}$ is integrable. Moreover for $X\in H^{\perp_g}$ we have
\begin{align*}
0=d(\alpha\wedge\beta)(X,F,\varphi F)
=g(F,F)d\alpha(X,F)+g(F,F)d\beta(X,\varphi F).
\end{align*}
Moreover, using the fact that $\nabla$ is Levi-Civita connection for $g$ and $g(\varphi F,\varphi F)=g(F,F)$,we obtain
\begin{align*}
d\alpha(X,F)=\frac{1}{2} X(g(F,F))-g(\nabla_FF,X),\\
d\beta(X,\varphi F)=\frac{1}{2} X(g(F,F))-g(\nabla_{\varphi F}\varphi F,X).
\end{align*}
Therefore
\begin{align*}
0=g(F,F)d\alpha(X,F)+g(F,F)d\beta(X,\varphi F)=g(F,F)(X(g(F,F))\\-g(X,\nabla_FF+\nabla_{\varphi F}\varphi F))
\end{align*}
but $g(F,F)\neq 0$ so
$$
X(g(F,F))=g(X,\nabla_FF+\nabla_{\varphi F}\varphi F)
$$
for all  $X\in H^{\perp_g}$.
\end{proof}

\begin{ex}
Let $M=\Bigl\{(x_1,y_1,\dots,x_n,y_n,t)\in \mathbb{R}^{2n+1}:y_1>0\Bigr\}$.
Define on $M$ an almost contact pseudo-metric structure $(\varphi,\xi,\eta,g)$ by
$$
\xi=\partial_t,\qquad \eta=dt,\qquad \varphi \xi=0,
$$
$$
\varphi \partial_{x_k}=\partial_{y_k},\qquad \varphi \partial_{y_k}=-\partial_{x_k},
\qquad k=1,\dots,n,
$$
and
$$
g=-dx_1^2-dy_1^2+\sum_{k=2}^n (dx_k^2+dy_k^2)+dt^2.
$$
Then $(M,\varphi,\xi,\eta,g)$ is a co-K\"ahler pseudo-metric manifold of index $2$.

Consider the vector field $F=\frac{1}{y_1}\partial_{x_1}+\partial_{y_1}$.
Then $F\in \ker\eta$ and
$$
\varphi F=\frac{1}{y_1}\partial_{y_1}-\partial_{x_1}.
$$
Moreover
$$
g(F,F)=g(\varphi F,\varphi F)=-\frac{1}{y_1^2}-1<-1.
$$
Hence, the metric
$$
g^\ast(X,Y)=g(X,Y)+g(F,X)g(F,Y)+g(\varphi F,X)g(\varphi F,Y)
$$
is nondegenerate and $\ind(g^\ast)=\ind(g)-2=0$ hence Riemannian.
A direct computation gives

$$
g^\ast=\frac{1}{y_1^2}(dx_1^2+dy_1^2)+\sum_{k=2}^n (dx_k^2+dy_k^2)+dt^2.
$$
Since
$$
\Phi^\ast=\Phi-g(F,\cdot)\wedge g(\varphi F,\cdot),
$$
it follows that $d\eta=0$ and $d\Phi^\ast=0$. Hence $(M,\varphi,\xi,\eta,g^\ast)$ is an almost co-K\"ahler metric manifold. In particular, $(M,g^\ast)$ is isometric to
$H^2\times \mathbb{R}^{2n-1}$, where $H^2=\{(x_1,y_1):y_1>0\}$ is with metric $\frac{1}{y_1^2}(dx_1^2+dy_1^2)$.
\end{ex}

\begin{ex}
Let $M=\mathbb{R}^{5}$ with coordinates $(x_{1},y_{1},x_{2},y_{2},t)$. Let $a$ be a nonconstant smooth function depending only on $x_{2},y_{2}$. Define an almost contact structure $(\varphi,\xi,\eta)$ on $M$ by
$$
\xi=\partial_{t},\qquad \eta=dt,\qquad \varphi\xi=0,
$$
$$
\varphi(\partial_{x_{1}})=\partial_{y_{1}}+a\partial_{x_{1}},\qquad
\varphi(\partial_{y_{1}})=-(1+a^{2})\partial_{x_{1}}-a\partial_{y_{1}},
$$
$$
\varphi(\partial_{x_{2}})=\partial_{y_{2}},\qquad
\varphi(\partial_{y_{2}})=-\partial_{x_{2}}.
$$
Consider the Riemannian metric
$$
g=dx_{1}^{2}-2adx_{1}dy_{1}+(1+a^{2})dy_{1}^{2}+dx_{2}^{2}+dy_{2}^{2}+dt^{2}.
$$
Then $(M,\varphi,\xi,\eta,g)$ is an almost co-K\"ahler metric manifold. Indeed, its fundamental $2$-form is
$$
\Phi=-dx_{1}\wedge dy_{1}-dx_{2}\wedge dy_{2},
$$
hence $d\eta=0$ and $d\Phi=0$. Observe that neither $g$ nor $\varphi$ is a direct product, since $a$ is nonconstant.

Let now $b$ be a smooth function depending only on $x_{1},y_{1}$ and consider the vector field $F=b\partial_{x_{1}}$.
Then $F\in \ker\eta$ and $\varphi F=b\bigl(\partial_{y_{1}}+a\partial_{x_{1}}\bigr)$.
Direct computations give
$$
d\bigl(g(F,\cdot)\wedge g(\varphi F,\cdot)\bigr)=0.
$$
Also $H^{\perp_{g}}=\operatorname{span}(F,\varphi F)^{\perp_{g}}=\operatorname{span}\{\partial_{x_{2}},\partial_{y_{2}},\xi\}$,
so $H^{\perp_{g}}$ is integrable. The metric

$$
g^{\ast}=(1+b^{2})\bigl(dx_{1}^{2}-2adx_{1}dy_{1}+(1+a^{2})dy_{1}^{2}\bigr)+dx_{2}^{2}+dy_{2}^{2}+dt^{2}.
$$ is again almost co-K\"ahler.
Thus $(M,\varphi,\xi,\eta,g^{\ast})$ is an almost co-K\"ahler metric manifold which, in general, is not a direct product.
\end{ex}

\section{Local product structure (almost contact $\times$ almost complex) and other geometric properties}

\begin{lm}\label{lm::productsplit::acx}
Let $(M_1,\varphi_1,\xi_1,\eta_1)$ be an almost contact manifold and let
$(M_2,J_2)$ be an almost complex manifold. Put
$$
M:=M_1\times M_2
$$
and define an almost contact structure $(\varphi,\xi,\eta)$ on $M$ by
$$
\varphi(X_1,X_2)=\bigl(\varphi_1X_1,J_2X_2\bigr),\qquad
\xi=(\xi_1,0),\qquad
\eta(X_1,X_2)=\eta_1(X_1),
$$
where $X_1\in TM_1$ and $X_2\in TM_2$.

Let $g$ be a pseudo-Riemannian metric on $M$ such that the canonical foliations
$$
M_1\times\{q\},\qquad \{p\}\times M_2
$$
are nondegenerate and orthogonal.
If $(M,\varphi,\xi,\eta,g,\eps)$ is an almost co-K\"ahler pseudo-metric
manifold, then there exist pseudo-Riemannian metrics $g_1$ on $M_1$ and $g_2$ on $M_2$
such that
$$
g=g_1+g_2
$$
and
$(M_1,\varphi_1,\xi_1,\eta_1,g_1,\eps)$
is an almost co-K\"ahler pseudo-metric manifold and
$(M_2,J_2,g_2)$
is an almost pseudo-K\"ahler manifold.

Moreover, if $(M,\varphi,\xi,\eta,g,\eps)$ is a co-K\"ahler pseudo-metric
manifold, then $(M_1,\varphi_1,\xi_1,\eta_1,g_1,\eps)$
is a co-K\"ahler pseudo-metric manifold and
$(M_2,J_2,g_2)$
is a pseudo-K\"ahler manifold.
\end{lm}

\begin{proof}
Since $d\eta=0$, we immediately get $d\eta_1=0$. Direct computation in local frames yields
 that $g$ is
the direct product of the pseudo-metrics $g_1:=g|_{M_1\times\{q\}}$ and $g_2:=g|_{\{p\}\times M_2}$.
\par We shall show now that $(M_1,\varphi_1,\xi_1,\eta_1,g_1,\eps)$ is an almost co-K\"ahler pseudo-metric manifold. Indeed, for $X,Y\in TM_1$ we compute
\begin{align*}
g_1(\varphi_1X,\varphi_1Y)
&=g\bigl((\varphi_1X,0),(\varphi_1Y,0)\bigr)\\
&=g\bigl(\varphi(X,0),\varphi(Y,0)\bigr)\\
&=g\bigl((X,0),(Y,0)\bigr)-\eps\,\eta(X,0)\eta(Y,0)\\
&=g_1(X,Y)-\eps\,\eta_1(X)\eta_1(Y).
\end{align*}
We also have
$$
g_1(X,\xi_1)=g\bigl((X,0),(\xi_1,0)\bigr)=\eps\,\eta(X,0)=\eps\,\eta_1(X),
$$
so $(M_1,\varphi_1,\xi_1,\eta_1,g_1,\eps)$ is an almost contact pseudo-metric manifold.
Let $\Phi_1$ be  a fundamental form for $g_1$, that is for $X,Y\in TM_1$
$$
\Phi_1(X,Y):=g_1(X,\varphi_1Y)=\Phi\bigl((X,0),(Y,0)\bigr).
$$
Therefore, for all $X,Y,Z\in TM_1$,
$$
d\Phi_1(X,Y,Z)=d\Phi\bigl((X,0),(Y,0),(Z,0)\bigr)=0.
$$
Since also $d\eta_1=0$, we conclude that $(M_1,\varphi_1,\xi_1,\eta_1,g_1,\eps)$ is an almost co-K\"ahler pseudo-metric manifold.
\par Similarly, for $U,V\in TM_2$,
\begin{align*}
g_2(J_2U,J_2V)
&=g\bigl((0,J_2U),(0,J_2V)\bigr)=g\bigl(\varphi(0,U),\varphi(0,V)\bigr)\\
&=g\bigl((0,U),(0,V)\bigr)=g_2(U,V).
\end{align*}
Thus $(M_2,J_2,g_2)$ is an almost pseudo-Hermitian manifold. If $\Omega_2$ is a fundamental form for $g_2$, then for all $U,V,W\in TM_2$
$$
d\Omega_2(U,V,W)=d\Phi\bigl((0,U),(0,V),(0,W)\bigr)=0
$$
so $(M_2,J_2,g_2)$ is an almost pseudo-K\"ahler manifold.

Finally, assume that $(M,\varphi,\xi,\eta,g,\eps)$ is co-K\"ahler. Since $d\eta=0$, normality of $(\varphi,\xi,\eta)$ is equivalent to $[\varphi,\varphi]=0$.
By straightforward computations we obtain
$$
[\varphi_1,\varphi_1]=0
\qquad\text{and}\qquad
[J_2,J_2]=0.
$$
Thus $(\varphi_1,\xi_1,\eta_1)$ is normal and $J_2$ is integrable, what completes the proof of last part.
\end{proof}

\begin{df}\label{df::phihol}
Let $(M,\varphi,\xi,\eta)$ be an almost contact manifold. A vector field
$F\in\XM$ will be called \emph{$\varphi$-holomorphic} if
$$
L_F\varphi=0.
$$
\end{df}

\begin{uw}\label{uw::phihol}
In the paper \cite{BS}, on a normal almost contact manifold
$(M,\varphi,\xi,\eta)$, a local vector field $F$ is called
\emph{contact-holomorphic} if
$$
(L_F\varphi)(Y)=\eta([F,\varphi Y])\xi
$$
for every $Y\in\XM$.
In the present paper we use the stronger condition
$$
L_F\varphi=0
$$
and call such a vector field $\varphi$-holomorphic. However, if
$$
d\eta=0
\qquad\text{and}\qquad
F\in\ker\eta,
$$
then $\eta([F,\varphi Y])=0$ for every $Y\in\XM$. Therefore, in this case, the notions of
contact-holomorphic and $\varphi$-holomorphic vector fields coincide.
\end{uw}

\begin{uw}\label{uw::phihol::cok}
If $(M,\varphi,\xi,\eta,g)$ is a co-K\"ahler pseudo-metric manifold, then
$\nabla\varphi=0$. Consequently, for every $X, F\in\XM$,
$$
L_F\varphi=0\quad \Leftrightarrow \quad \nabla_{\varphi X}F=\varphi(\nabla_XF).
$$
Thus, on a co-K\"ahler pseudo-metric manifold, Definition
\ref{df::phihol} is the direct analogue of the usual notion of a holomorphic
vector field on a K\"ahler manifold.
\end{uw}

\begin{df}\label{df::phiholmap}
Let $(M,\varphi,\xi,\eta)$ and $(\widetilde M,\widetilde\varphi,\widetilde\xi,\widetilde\eta)$
be almost contact manifolds. A smooth map
$$
\Psi:\widetilde M\to M
$$
will be called \emph{$\varphi$-holomorphic} if
\begin{equation}\label{eq::phiholmap}
d\Psi\circ \widetilde\varphi=\varphi\circ d\Psi,
\qquad
d\Psi(\widetilde\xi)=\xi,
\qquad
\widetilde\eta=\Psi^\ast\eta.
\end{equation}
If, moreover, $\Psi$ is a local diffeomorphism, then $\Psi$ will be called a
\emph{local $\varphi$-holomorphism}.
\end{df}

\begin{tw}\label{tw::localprod}
Let $(M,\varphi,\xi,\eta,g,\eps)$ be a co-K\"ahler pseudo-metric manifold and
let $F\in\ker\eta$ be a $\varphi$-holomorphic vector field.
Assume that $g(F_p,F_p)\neq -1$, $g(F_p,F_p)\neq 0$ for all $p\in M$.
If $(M,\varphi,\xi,\eta,g^\ast,\eps)$ is a co-K\"ahler pseudo-metric manifold, where
$
g^\ast$ is defined by Formula (\ref{eq::Gast}),
then $(M,\varphi,\xi,\eta,g,\eps)$ is locally isometric and $\varphi$-holomorphic
to
$$
\bigl(N\times \R^2,\widetilde\varphi,\widetilde\xi,\widetilde\eta,
g|_N+\psi(t,s)(dt^2+ds^2),\eps\bigr),
$$
where $N$ is a leaf of the distribution
$$
H^{\perp_g}=\Span\{F,\varphi F\}^{\perp_g},
$$
$\psi$ is a nowhere vanishing function, and
$(\widetilde\varphi,\widetilde\xi,\widetilde\eta)$ is the almost contact
structure on $N\times\R^2$ given by
\begin{align}
\label{eq::inducedvarphi} \widetilde\varphi(X+a\partial_t+b\partial_s)=\varphi_NX+a\partial_s-b\partial_t,\\
\label{eq::inducedxieta} \widetilde\xi=\xi_N,\qquad \widetilde\eta(X+a\partial_t+b\partial_s)=\eta_N(X)
\end{align}
where $X\in\X(N)$, $a,b\in\Cniesk(N\times\R^2)$ and $(\varphi_N,\xi_N,\eta_N,g_N)$ is induced structure on $N$.
Moreover, under this local identification, $F=\partial_t$, $\varphi F=\partial_s$.
\end{tw}

\begin{proof}
Since $(M,\varphi,\xi,\eta,g^\ast,\eps)$ is co-K\"ahler, Theorem \ref{lm::acokequiv} implies that the distribution $H^{\perp_g}=\Span\{F,\varphi F\}^{\perp_g}$
is integrable.
\par Fix a point $p_0\in M$ and let $N$ be the leaf of
$H^{\perp_g}$ through $p_0$.
Since $0=L_F\varphi=[F,\varphi F]$
we have that the distribution $\Span\{F,\varphi F\}$ is also integrable.
Since
$$
g(F,F)\neq 0,\qquad g(F,\varphi F)=0,\qquad g(\varphi F,\varphi F)=g(F,F)
$$
this distribution is nondegenerate. By definition, it is orthogonal to
$H^{\perp_g}$.
Let $\Sigma$ be the leaf of $\Span\{F,\varphi F\}$ through $p_0$. Since $[F,\varphi F]=0$
and $F,\varphi F$ are linearly independent, there exist local coordinates
$(t,s)$ on $\Sigma$ such that
$$
F=\partial_t,\qquad \varphi F=\partial_s.
$$
Because $H^{\perp_g}$ and $\Span\{F,\varphi F\}$ are complementary integrable
distributions, there exists a local diffeomorphism
$$
\Psi:N\times \R^2\ni (x,t,s)\mapsto \Phi_t^{F}\circ \Phi_s^{\varphi F}(x)\in M
$$
from a neighbourhood of $(p_0,0,0)$ onto a neighbourhood of $p_0$ such that
$$
d\Psi(\partial_t)=F,\qquad d\Psi(\partial_s)=\varphi F,
$$
where $\Phi^{F},\Phi^{\varphi F}$ are the flow of $F$ and $\varphi F$, respectively.
Since $\xi\in H^{\perp_g}$ and $\varphi(H^{\perp_g})\subset H^{\perp_g}$ we can restrict $(\varphi,\xi,\eta)$ to $N$, and obtain an almost contact structure
$(\varphi_N,\xi_N,\eta_N)$ on the leaf $N$.
 On the other hand,
$\Span\{F,\varphi F\}$ is $\varphi$-invariant and, with respect to the basis
$\{\partial_t,\partial_s\}$, the restriction of $\varphi$ to this distribution
is the standard complex structure on $\R^2$.

Let $(\widetilde\varphi,\widetilde\xi,\widetilde\eta)$ denotes the almost
contact structure on $N\times\R^2$ defined by \eqref{eq::inducedvarphi} and \eqref{eq::inducedxieta}, then
$$
d\Psi\circ \widetilde\varphi=\varphi\circ d\Psi,
\qquad
d\Psi(\widetilde\xi)=\xi,
\qquad
\widetilde\eta=\Psi^\ast\eta.
$$
That is, $\Psi$ is a local $\varphi$-holomorphism.
Moreover,
$$
(N\times\R^2,\widetilde\varphi,\widetilde\xi,\widetilde\eta,\Psi^\ast g,\eps)
$$
is a co-K\"ahler pseudo-metric manifold. Therefore, by
Lemma \ref{lm::productsplit::acx}, there exist pseudo-metrics $h_1$ on $N$ and
$h_2$ on $\R^2$ such that $\Psi^\ast g=h_1+h_2$, $(N,\varphi_N,\xi_N,\eta_N,h_1,\eps)$ is a co-K\"ahler pseudo-metric manifold
and $(\R^2,J_0,h_2)$ is a pseudo-K\"ahler manifold, where $J_0$ is the standard
complex structure on $\R^2$.
By construction $h_1=g|_N$. Since $h_2$ is compatible with the standard
complex structure $J_0$ on $\R^2$, we have $h_2(\partial_t,\partial_t)=h_2(\partial_s,\partial_s)$, $h_2(\partial_t,\partial_s)=0$.
Thus $h_2=\psi(t,s)(dt^2+ds^2)$ for a nowhere vanishing function $\psi$.
Therefore
$$
\Psi^\ast g=g|_N+\psi(t,s)(dt^2+ds^2),
$$
and $\Psi$ is a local isometry and a local $\varphi$-holomorphism.
\end{proof}

\begin{stw}\label{stw::difference}
Let $(M,\varphi,\xi,\eta,g,\eps)$ be an almost contact pseudo-metric manifold
and let $F\in\ker\eta$ be a vector field such that $g(F_p,F_p)\neq -1$ for all $p\in M$ and $g^\ast$ defined by Formula (\ref{eq::Gast}).

Let $\widehat{\nabla}$ and $\widehat{\nabla}^\ast$ denote the Levi-Civita
connections of $g$ and $g^\ast$, respectively, and define
$$
D(U,V):=\widehat{\nabla}^\ast_U V-\widehat{\nabla}_U V.
$$
Then $D$ is symmetric and
\begin{equation}\label{eq::difference}
\begin{aligned}
2g^\ast(D(U,V),W)
&=\alpha(W)(L_Fg)(U,V)+\alpha(V)d\alpha(U,W)+\alpha(U)d\alpha(V,W)\\
&\quad+\beta(W)(L_{\varphi F}g)(U,V)
+\beta(V)d\beta(U,W)+\beta(U)d\beta(V,W)
\end{aligned}
\end{equation}
for all $U,V,W\in\XM$, where 
$\alpha:=g(F,\cdot),\ \beta:=g(\varphi F,\cdot)
$.
Moreover, if $g(F_p,F_p)\neq 0$ for all $p\in M$, then
\begin{equation}\label{eq::divergence}
\operatorname{div}_{g^\ast}U=\operatorname{div}_gU+\frac{U(\lambda)}{1+\lambda},
\end{equation}
for every $U\in\XM$, where $\lambda:=g(F,F)$.
\end{stw}

\begin{proof}
The symmetry of $D$ is an immediate consequence of the fact that $\widehat{\nabla}^\ast$ and $\widehat{\nabla}$ are torsion-free connections.
\par Since both sides of \eqref{eq::difference} are tensorial, we may assume that the
Lie brackets of the vector fields under consideration vanish. Using the definition of $g^\ast$ and the Koszul
formula for $g^\ast$ and $g$, we obtain
$$
2g^\ast(\widehat{\nabla}^\ast_UV,W)
=2g(\widehat{\nabla}_UV,W)
+\alpha(V)d\alpha(U,W)+\alpha(U)d\alpha(V,W)
+\beta(V)d\beta(U,W)+$$ $$ +\beta(U)d\beta(V,W)+\alpha(W)\bigl(U\alpha(V)+V\alpha(U)\bigr)
+\beta(W)\bigl(U\beta(V)+V\beta(U)\bigr).
$$

Moreover,
$$
U\alpha(V)+V\alpha(U)
=(L_Fg)(U,V)+2g(F,\widehat{\nabla}_UV),
$$
and
$$
U\beta(V)+V\beta(U)
=(L_{\varphi F}g)(U,V)+2g(\varphi F,\widehat{\nabla}_UV).
$$
Therefore
\begin{align*}
2g^\ast(\widehat{\nabla}^\ast_UV,W)
&=2g^\ast(\widehat{\nabla}_UV,W)\\
&\quad+\alpha(W)(L_Fg)(U,V)+\alpha(V)d\alpha(U,W)+\alpha(U)d\alpha(V,W)\\
&\quad+\beta(W)(L_{\varphi F}g)(U,V)
+\beta(V)d\beta(U,W)+\beta(U)d\beta(V,W).
\end{align*}
Thus we obtain \eqref{eq::difference}.
\par Now assume that $g(F_p,F_p)\neq 0$ for all $p\in M$, and let $\lambda:=g(F,F)$.
Since $F\in\ker\eta$, we have $g(F,\varphi F)=0, g(\varphi F,\varphi F)=g(F,F)=\lambda$.
We have $TM=H^{\perp_g}\oplus \Span\{F,\varphi F\}$, where $H^{\perp_g}=\Span\{F,\varphi F\}^{\perp_g}$ and $g^\ast(F,F)=g^\ast(\varphi F,\varphi F)=\lambda(1+\lambda)$.
Using  \eqref{eq::difference} we get $g^\ast(D(X,U),X)=0$ for $X\in H^{\perp_g}$.
Since $g^\ast(\cdot,X)=g(\cdot,X)$ for every $X\in H^{\perp_g}$ we also obtain $g(D(X,U),X)=0$.
Moreover, using again \eqref{eq::difference} we get
\begin{equation}\label{eq::DF}
g^\ast(D(F,U),F)=\frac{\lambda}{2}U(\lambda)
\end{equation}
and
\begin{equation}\label{eq::DphiF}
g^\ast(D(\varphi F,U),\varphi F)=\frac{\lambda}{2}U(\lambda).
\end{equation}
Let
$$
e_1,\dots,e_{2n-1}
$$
be a local $g$-pseudo-orthonormal basis of $H^{\perp_g}$ at a point and let $\eps_i:=g(e_i,e_i)\in\{-1,1\}$.
Then
$$
e_1,\dots,e_{2n-1},\frac{1}{\sqrt{|\lambda(1+\lambda)|}}F,\frac{1}{\sqrt{|\lambda(1+\lambda)|}}\varphi F
$$
is a local $g^\ast$-pseudo-orthonormal basis of $TM$. Therefore
\begin{align*}
\operatorname{div}_{g^\ast}U
&=\sum_{i=1}^{2n-1}\eps_i\,g^\ast(\widehat{\nabla}^\ast_{e_i}U,e_i)
+\frac{1}{\lambda(1+\lambda)}g^\ast(\widehat{\nabla}^\ast_FU,F)
+\frac{1}{\lambda(1+\lambda)}g^\ast(\widehat{\nabla}^\ast_{\varphi F}U,\varphi F)\\
&=\sum_{i=1}^{2n-1}\eps_i\,g(\widehat{\nabla}_{e_i}U,e_i)
+\frac{1}{\lambda}g(\widehat{\nabla}_FU,F)
+\frac{1}{\lambda}g(\widehat{\nabla}_{\varphi F}U,\varphi F)\\
&\quad+\frac{1}{\lambda(1+\lambda)}g^\ast(D(F,U),F)
+\frac{1}{\lambda(1+\lambda)}g^\ast(D(\varphi F,U),\varphi F)\\
&=\operatorname{div}_gU+\frac{U(\lambda)}{1+\lambda},
\end{align*}
where in the last step we used \eqref{eq::DF} and \eqref{eq::DphiF}. This proves
\eqref{eq::divergence}.
\end{proof}

\begin{wn}\label{wn::secondfund}
Let $(M,\varphi,\xi,\eta,g,\eps)$ be an almost contact pseudo-metric manifold
and let $F\in\ker\eta$ be a vector field such that $g(F_p,F_p)\neq -1$, $g(F_p,F_p)\neq 0$ for all $p\in M$. Let
$
g^\ast$ be defined by Formula (\ref{eq::Gast})
and
$$
H^{\perp_g}:=\Span\{F,\varphi F\}^{\perp_g}.
$$
Assume that $H^{\perp_g}$ is an integrable distribution. If $h$ and $h^\ast$
denote the second fundamental forms of the leaves of $H^{\perp_g}$ in
$(M,g)$ and $(M,g^\ast)$, respectively, then
$$
h^\ast(X,Y)=\frac{1}{1+\lambda}h(X,Y)
$$
for all $X,Y\in\Gamma(H^{\perp_g})$, where $\lambda:=g(F,F)$.
\end{wn}
\begin{proof}
We see that for every $X\in\Gamma(H^{\perp_g})$ we have $g^\ast(X,F)=g(X,F)=0$, $g^\ast(X,\varphi F)=g(X,\varphi F)=0$.
Thus the leaves of $H^{\perp_g}$ have the same normal bundle with respect to $g$ and $g^\ast$ $H:=\Span\{F,\varphi F\}$.
Let $X,Y\in\Gamma(H^{\perp_g})$. Since
$$
\widehat{\nabla}^\ast_XY=\widehat{\nabla}_XY+D(X,Y),
$$
we get
$$
h^\ast(X,Y)-h(X,Y)=D(X,Y)^\perp,
$$
where ${}^\perp$ denotes the projection onto $H$.
Now taking $U=X$, $V=Y$ and $W=F$ in \eqref{eq::difference}, since $\alpha(X)=\alpha(Y)=\beta(X)=\beta(Y)=0$ and $\widehat{\nabla}g=0$ we obtain
$$
2g^\ast(D(X,Y),F)=\lambda(L_Fg)(X,Y)=\lambda \bigl(g(\widehat{\nabla}_XF,Y)+g(X,\widehat{\nabla}_YF)\bigr).
$$
Since $g(F,Y)=g(F,X)=0$ we have
$$
g(\widehat{\nabla}_XF,Y)=-g(\widehat{\nabla}_XY,F),
\qquad
g(X,\widehat{\nabla}_YF)=-g(\widehat{\nabla}_YX,F).
$$
Since $h$ is symmetric we obtain $(L_Fg)(X,Y)=-2g(h(X,Y),F)$.
Therefore
\begin{equation}\label{eq::hF}
g^\ast(D(X,Y),F)=-\lambda\,g(h(X,Y),F).
\end{equation}
In similar way, using \eqref{eq::difference}, we get
\begin{equation}\label{eq::hphiF}
g^\ast(D(X,Y),\varphi F)=-\lambda\,g(h(X,Y),\varphi F).
\end{equation}
On the other hand
\begin{align*}
g^\ast(h^\ast(X,Y),F)-g^\ast(h(X,Y),F)
&=g^\ast(D(X,Y),F),\\
g^\ast(h^\ast(X,Y),\varphi F)-g^\ast(h(X,Y),\varphi F)
&=g^\ast(D(X,Y),\varphi F).
\end{align*}
Moreover we have
$$
g^\ast(h(X,Y),F)=(1+\lambda)g(h(X,Y),F),
$$
and
$$
g^\ast(h(X,Y),\varphi F)=(1+\lambda)g(h(X,Y),\varphi F).
$$

Since $F$ and $\varphi F$ span the normal bundle of the leaves of
$H^{\perp_g}$ using  \eqref{eq::hF} and \eqref{eq::hphiF} we get
$$
h^\ast(X,Y)=\frac{1}{1+\lambda}h(X,Y).
$$
\end{proof}

\begin{uw}\label{uw::secondfund}
Under the assumptions of Corollary \ref{wn::secondfund}, we have
$$
h(X,Y)=\frac{1}{\lambda}\Bigl(-g(\widehat{\nabla}_XF,Y)\,F-g(\widehat{\nabla}_X\varphi F,Y)\,\varphi F\Bigr)
$$
for all $X,Y\in\Gamma(H^{\perp_g})$.
If moreover $(M,\varphi,\xi,\eta,g,\eps)$ is co-K\"ahler then $\widehat{\nabla}\varphi=0$ and hence $\widehat{\nabla}_X\varphi F=\varphi(\widehat{\nabla}_XF)$.
Therefore
$$
h(X,Y)=\frac{1}{\lambda}\Bigl(-g(\widehat{\nabla}_XF,Y)\,F+g(\widehat{\nabla}_XF,\varphi Y)\,\varphi F\Bigr).
$$
\end{uw}

\begin{lm}\label{lm::cok::killing}
Let $(M,\varphi,\xi,\eta,g)$ be a compact co-K\"ahler metric manifold and let
$U\in\XM$ be a $\varphi$-holomorphic vector field such that
$$
\operatorname{div}_gU=0.
$$
Then $U$ is Killing, that is,
$$
L_Ug=0.
$$
\end{lm}
\begin{proof}
By the structure theorem for compact co-K\"ahler manifolds, see \cite{BO},
there exists a finite covering $\pi:K\times S^1\to M$
such that $K$ is a compact K\"ahler manifold and the pull-back co-K\"ahler
structure on $K\times S^1$ is the standard product one. Denote it by $(\bar\varphi,\bar\xi,\bar\eta,\bar g)$,
so that
$$
\bar\xi=\partial_t,\qquad \bar\eta=dt, \quad \bar\varphi(X+a\partial_t)=JX
$$
for every $X\in\X(K)$ and $a\in\Cniesk(K\times S^1)$, where $J$ is the complex
structure on $K$.
Let $\bar U$ be the lift of $U$ to $K\times S^1$. Since $\pi$ is a local
isometry preserving the structure tensors, $\bar U$ is $\bar\varphi$-holomorphic
and $\operatorname{div}_{\bar g}\bar U=0$.
Let us denote
$$
\bar U=X+a\,\partial_t,
$$
where $X$ is tangent to $K$ and $a\in\Cniesk(K\times S^1)$.
\par Since $K\times S^1$ is co-K\"ahler, we have $\bar\nabla\bar\varphi=0$.
Hence, by the characterization of $\varphi$-holomorphic vector fields,
$$
\bar\nabla_{\bar\varphi Z}\bar U=\bar\varphi(\bar\nabla_Z\bar U)
$$
for every $Z\in\X(K\times S^1)$.
Take $Z=Y\in\X(K)$. Since the Levi-Civita connection of $\bar g$ is the product
connection, we obtain
$$
\bar\nabla_Y\bar U=\nabla^K_YX+Y(a)\partial_t,\quad \bar\nabla_{JY}\bar U=\nabla^K_{JY}X+JY(a)\partial_t.
$$
Therefore
$$
\nabla^K_{JY}X+JY(a)\partial_t=J(\nabla^K_YX).
$$
Comparing the horizontal and vertical parts, we get
$$
\nabla^K_{JY}X=J(\nabla^K_YX),
\qquad
JY(a)=0
$$
for every $Y\in\X(K)$. Since $J$ is an automorphism of $TK$, it follows that $Y(a)=0$
for every $Y\in\X(K)$, so $a$ depends only on the $S^1$-variable.
Now take $Z=\partial_t$.
Since $\bar\varphi(\partial_t)=0$,
the relation
$$
\bar\nabla_{\bar\varphi Z}\bar U=\bar\varphi(\bar\nabla_Z\bar U)
$$
gives
$$
0=\bar\varphi(\bar\nabla_{\partial_t}\bar U).
$$
But
$$
\bar\nabla_{\partial_t}\bar U=\bar\nabla_{\partial_t}X+\partial_t(a)\partial_t,
$$
hence $J(\bar\nabla_{\partial_t}X)=0$.
Thus $\bar\nabla_{\partial_t}X=0$.
So $X$ is independent of $t$.
\par The relation
$$
\nabla^K_{JY}X=J(\nabla^K_YX)
$$
is equivalent to $L_XJ=0$,
hence $X$ is holomorphic on the compact K\"ahler manifold $(K,J,g_K)$.
Moreover, since $\bar g=g_K+dt^2$, we have
$$
\operatorname{div}_{\bar g}\bar U=\operatorname{div}_{g_K}X+\partial_t(a).
$$
Because $X$ is independent of $t$ and $a$ depends only on $t$, the equality
$$
\operatorname{div}_{\bar g}\bar U=0
$$
implies that both $\operatorname{div}_{g_K}X$ and $\partial_t(a)$ are constant.
As $a$ is defined on $S^1$, we obtain $\partial_t(a)=0$ (smooth function on a circle cannot have a nonzero constant derivative),
and therefore $\operatorname{div}_{g_K}X=0$.
\par By the classical K\"ahlerian result, see \cite[Theorem 4.81]{Ba}, every
holomorphic divergence-free vector field on a compact K\"ahler manifold is
Killing. Hence $X$ is Killing on $K$. Since $a$ is constant and $\partial_t$ is
Killing on $K\times S^1$, the vector field
$$
\bar U=X+a\,\partial_t
$$
is Killing with respect to $\bar g$ ($L_{\bar U}\bar g=0$).
Finally, we conclude that
$L_Ug=0$ so $U$ is Killing.
\end{proof}

\begin{tw}\label{tw::radical}
Let $(M,\varphi,\xi,\eta,g,\eps)$ be a compact co-K\"ahler pseudo-metric
manifold of dimension $2n+1>5$ and index $2$, furnished with a timelike vector
field $E\in\ker\eta$ such that the distributions
$$
\Span\{E,\varphi E\}
\qquad\text{and}\qquad
H^{\perp_g}:=\Span\{E,\varphi E\}^{\perp_g}
$$
are integrable. If $U\in\XM$ is a $\varphi$-holomorphic vector field such that
$$
\operatorname{div}_gU=0,
$$
then the dimension of the radical of the symmetric bilinear form $L_Ug$ is at least
$$
2n-3.
$$
\end{tw}

\begin{proof}
Let
$$
F:=\frac{2}{\sqrt{-g(E,E)}}E.
$$
Then $F\in\ker\eta$, $g(F,F)=-4$ and
$\Span\{F,\varphi F\}=\Span\{E,\varphi E\}$, $\Span\{F,\varphi F\}^{\perp_g}=H^{\perp_g}$.
Consider the metric
$
g^\ast(X,Y)$ defined by Formula (\ref{eq::Gast}).
Since $g(F,F)=-4$, Proposition \ref{lem::indeks:gstar} implies that $\ind g^\ast=\ind g-2=0$ thus
$g^\ast$ is Riemannian.
\par We shall show that $(M,\varphi,\xi,\eta,g^\ast,\eps)$ is co-K\"ahler. Since $(\varphi,\xi,\eta)$ is normal Theorem \ref{lm::acokequiv} implies that
it is enough to verify the condition
$$
X(g(F,F))=g(X,\nabla_FF+\nabla_{\varphi F}\varphi F)
$$
for every $X\in\Gamma(H^{\perp_g})$.
The left-hand side vanishes, because $g(F,F)$ is constant. For the right-hand side,
using $\nabla\varphi=0$ and $\eta(\nabla_FF)=0$, we obtain
\begin{align*}
g(X,\nabla_FF+\nabla_{\varphi F}\varphi F)
=g(\varphi X,[F,\varphi F]).
\end{align*}
Since $\Span\{F,\varphi F\}$ is integrable, we have $[F,\varphi F]\in\Span\{F,\varphi F\}$ and in consequence $g(\varphi X,[F,\varphi F])=0$
for all $X\in\Gamma(H^{\perp_g})$. Thus $(M,\varphi,\xi,\eta,g^\ast,\eps)$ is co-K\"ahler.
\par Now, since $g(F,F)=-4$, Proposition \ref{stw::difference} implies
$$
\operatorname{div}_{g^\ast}U=\operatorname{div}_gU=0.
$$
Since $(M,g^\ast)$ is compact and Riemannian,
Lemma \ref{lm::cok::killing} implies that $U$ is Killing with respect to $g^\ast$,
that is, $L_Ug^\ast=0$.
\par Let
$$
\alpha:=g(F,\cdot),
\qquad
\beta:=g(\varphi F,\cdot).
$$
Take $X,Y\in\Gamma(H^{\perp_g})$. Then
$$
\alpha(X)=\alpha(Y)=\beta(X)=\beta(Y)=0.
$$
Moreover, since $H^{\perp_g}$ is integrable, $[X,Y]\in\Gamma(H^{\perp_g})$
and
$$
d\alpha(X,Y)=0,
\qquad
d\beta(X,Y)=0.
$$
Therefore, by \eqref{eq::difference},
$$
g^\ast(D(X,U),Y)=g^\ast(D(Y,U),X)=0.
$$
Since $g^\ast=g$ on $H^{\perp_g}$ and
$$
L_Ug^\ast(X,Y)=L_Ug(X,Y)+g^\ast(D(X,U),Y)+g^\ast(D(Y,U),X),
$$
we conclude that $(L_Ug)(X,Y)=0$ for all $X,Y\in\Gamma(H^{\perp_g})$.
Using again formula \eqref{eq::difference} we get
$$
2g^\ast\bigl(D(F,U),F\bigr)=\alpha(F)\,U\bigl(g(F,F)\bigr),
\quad
2g^\ast\bigl(D(\varphi F,U),\varphi F\bigr)
=\beta(\varphi F)\,U\bigl(g(\varphi F,\varphi F)\bigr).
$$
Now, since $g(F,F)$ is constant we obtain
$$
g^\ast(D(F,U),F)=0,
\qquad
g^\ast(D(\varphi F,U),\varphi F)=0.
$$
Hence
\begin{align*}
0=L_Ug^\ast(F,F)&=(L_Ug)(F,F)+2g^\ast(D(F,U),F),\\
0=L_Ug^\ast(\varphi F,\varphi F)&=(L_Ug)(\varphi F,\varphi F)
+2g^\ast(D(\varphi F,U),\varphi F),
\end{align*}
and thus
$$
(L_Ug)(F,F)=0,
\qquad
(L_Ug)(\varphi F,\varphi F)=0.
$$

Since $\ind g=2$ and
$$
g(F,F)=g(\varphi F,\varphi F)=-4,
\qquad
g(F,\varphi F)=0,
$$
the distribution $H^{\perp_g}$ is Riemannian of dimension $2n-1$. Let
$$
e_1,\dots,e_{2n-1}
$$
be a local $g$-orthonormal frame of $H^{\perp_g}$. Then the matrix of $L_Ug$
with respect to the basis
$$
\Bigl\{e_1,\dots,e_{2n-1},\frac12F,\frac12\varphi F\Bigr\}
$$
has the form
$$
\begin{pmatrix}
0 & \cdots & 0 & \star & \star\\
\vdots & \ddots & \vdots & \vdots & \vdots\\
0 & \cdots & 0 & \star & \star\\
\star & \cdots & \star & 0 & \star\\
\star & \cdots & \star & \star & 0
\end{pmatrix}.
$$
Hence
$$
\operatorname{rank}(L_Ug)\le 4.
$$
Therefore
$$
\dim\Rad(L_Ug)\ge (2n+1)-4=2n-3.
$$
\end{proof}

\begin{tw}\label{tw::phisectional}
Let $(M,\varphi,\xi,\eta,g,\eps)$ be a co-K\"ahler pseudo-metric manifold and
let $F\in\ker\eta$ be a vector field such that
$$
g(F_p,F_p)\neq -1
$$
for all $p\in M$.
Let
$
g^\ast$ be defined by Formula (\ref{eq::Gast}).
If $\Pi\subset\ker\eta$ is a nondegenerate $\varphi$-invariant plane orthogonal
to $F$, then its sectional curvatures $K^\ast(\Pi)$ and $K(\Pi)$ with respect
to $g^\ast$ and $g$ are related by
\begin{align*}
K^\ast(\Pi)=K(\Pi)
&-\frac{3\lambda}{4(1+\lambda)}
\Bigl((\operatorname{div}_\Pi F)^2+d\alpha(X,\varphi X)^2\Bigr)\\
&-\frac{1}{1+\lambda}\Bigl(
g(\nabla_XF,X)^2+g(\nabla_{\varphi X}F,\varphi X)^2\\
&\qquad\qquad\qquad
+g(\nabla_XF,\varphi X)^2+g(\nabla_{\varphi X}F,X)^2
\Bigr).
\end{align*}
where $\lambda:=g(F,F)$, $X\in\Pi$ is such that
$g(X,X)=\pm1$, $\Pi=\Span\{X,\varphi X\}$
and
$$
\operatorname{div}_\Pi F:=g(\nabla_XF,X)+g(\nabla_{\varphi X}F,\varphi X).
$$
In particular, if $g(F,F)\geq 0$, then $K^\ast(\Pi)\leq K(\Pi)$.
\end{tw}
\begin{proof}
Choose $X\in\Pi$ such that $g(X,X)=\pm 1$
and set $Y:=\varphi X$.
Since $\Pi\subset\ker\eta$ is $\varphi$-invariant, we have
$$
\Pi=\Span\{X,Y\},
\qquad
g(Y,Y)=g(X,X),
\qquad
g(X,Y)=0.
$$
Since $\Pi$ is orthogonal to $F$ we easily get
$$
g(X,F)=g(Y,F)=g(X,\varphi F)=g(Y,\varphi F)=0.
$$
In consequence $g^\ast=g$ on $\Pi$. Now we have,
\begin{equation}\label{eq::phisectional::1}
K^\ast(\Pi)-K(\Pi)
=
g\bigl((\widehat{R}^\ast_{XY}Y-\widehat{R}_{XY}Y),X\bigr).
\end{equation}
The following general formula holds:
$$
\widehat{R}^\ast_{UV}W
=
\widehat{R}_{UV}W
+(\widehat{\nabla}_UD)(V,W)-(\widehat{\nabla}_VD)(U,W)
+D(U,D(V,W))-D(V,D(U,W)).
$$
Now \eqref{eq::phisectional::1} takes the form:
\begin{align*}\label{eq::phisectional::1b}
K^\ast(\Pi)-K(\Pi)
&=
g\bigl((\widehat{\nabla}_XD)(Y,Y),X\bigr)
-g\bigl((\widehat{\nabla}_YD)(X,Y),X\bigr)\\
\nonumber &\quad
+g\bigl(D(X,D(Y,Y)),X\bigr)
-g\bigl(D(Y,D(X,Y)),X\bigr).
\end{align*}
Let us denote
\begin{align*}
\lambda&:=g(F,F),\quad \alpha:=g(F,\cdot),\quad \beta:=g(\varphi F,\cdot),\\
a&:=g(\widehat{\nabla}_XF,X),\quad b:=g(\widehat{\nabla}_XF,Y),\quad
c:=g(\widehat{\nabla}_YF,X),\quad d:=g(\widehat{\nabla}_YF,Y).
\end{align*}
Then
$
\operatorname{div}_\Pi F=a+d$ and $ d\alpha(X,Y)=b-c.
$
Then
Proposition \ref{stw::difference} implies that $D(X,X)$, $D(Y,Y)$ and $D(X,Y)$
are sections of $\Span\{F,\varphi F\}$.
We have
\begin{align*}
g\bigl((\widehat{\nabla}_XD)(Y,Y),X\bigr)
&=
-g^\ast\bigl(D(Y,Y),\widehat{\nabla}_XX\bigr)
+g\bigl(D(Y,Y),F\bigr)g\bigl(\widehat{\nabla}_XX,F\bigr)\\
&\quad
+g\bigl(D(Y,Y),\varphi F\bigr)g\bigl(\widehat{\nabla}_XX,\varphi F\bigr)
-2g^\ast\bigl(D(\widehat{\nabla}_XY,Y),X\bigr),
\end{align*}

$$
g\bigl((\widehat{\nabla}_XD)(Y,Y),X\bigr)
=
-a^2-b^2+\frac{\lambda}{1+\lambda}(bc-ad),
$$
and
$$
g\bigl((\widehat{\nabla}_YD)(X,Y),X\bigr)
=
c^2+d^2+\frac{\lambda}{2(1+\lambda)}(ad-bc-c^2-d^2).
$$
Moreover, 
$
g\bigl(D(X,D(Y,Y)),X\bigr)=0
$
and
\begin{equation}\label{eq::phisectional::7}
g\bigl(D(Y,D(X,Y)),X\bigr)
=
\frac{\lambda}{4(1+\lambda)}(c^2-b^2+d^2-a^2).
\end{equation}
Taking all above together
 we obtain
$$
K^\ast(\Pi)-K(\Pi)
=
-\frac{3\lambda}{4(1+\lambda)}
\Bigl((\operatorname{div}_\Pi F)^2+d\alpha(X,Y)^2\Bigr)
-\frac{1}{1+\lambda}\Bigl(
a^2+b^2+c^2+d^2
\Bigr)
$$
and finally
\begin{align*}
K^\ast(\Pi)&-K(\Pi)= -\frac{3\lambda}{4(1+\lambda)}
\Bigl((\operatorname{div}_\Pi F)^2+d\alpha(X,\varphi X)^2\Bigr)\\
&-\frac{1}{1+\lambda}\Bigl(
g(\widehat{\nabla}_XF,X)^2+g(\widehat{\nabla}_{\varphi X}F,\varphi X)^2+g(\widehat{\nabla}_XF,\varphi X)^2+g(\widehat{\nabla}_{\varphi X}F,X)^2
\Bigr).
\end{align*}
\par If $g(F,F)=\lambda\geq 0$, then $\lambda+1>0$ and
$$
K^\ast(\Pi)\leq K(\Pi).
$$
\end{proof}
By direct computations using the above theorem we obtain the following corollary.
\begin{wn}\label{wn::phisectional}
Let $(M,\varphi,\xi,\eta,g,\eps)$ be a co-K\"ahler pseudo-metric manifold and
let $F\in\ker\eta$ be a vector field such that
$$
g(F_p,F_p)\neq -1
$$
for all $p\in M$. Suppose that $\Pi\subset\ker\eta$ is a nondegenerate
$\varphi$-invariant plane at a point $p\in M$ orthogonal to $F_p$. If $F$ is
Killing or $\varphi$-holomorphic and
$$
g(F_p,F_p)=-\frac{2}{3},
$$
then the sectional curvatures of $\Pi$ computed with respect to $g$ and $g^\ast$
coincide, that is,
$$
K^\ast(\Pi)=K(\Pi).
$$
\end{wn}

\section{Statistical structures}
 Statistical structures in contact-
type geometry have been investigated in \cite{AM},\cite{FHOS}, \cite{YKM}. Our aim in this section is to study the canonical variation of almost contact metric structures in this framework.
\begin{df}\label{df::compat}
Let $(M,\varphi,\xi,\eta,g,\eps)$ be an almost contact pseudo-metric manifold
and let $\nabla^s$ be a connection on $M$. We say that $\nabla^s$ is
\emph{compatible with the almost contact pseudo-metric structure}
$(\varphi,\xi,\eta,g,\eps)$ if
$$
\nabla^s\Phi=0 \qquad \text{and} \qquad \nabla^s\eta=0,
$$
where $\Phi(X,Y)=g(X,\varphi Y)$.
\end{df}

\begin{uw}\label{uw::LC::compat}
Let $(M,\varphi,\xi,\eta,g,\eps)$ be an almost contact pseudo-metric manifold
and let $\widehat{\nabla}$ denote the Levi-Civita connection of $g$. Since $\widehat{\nabla}g=0$,
we have $(\widehat{\nabla}_X\Phi)(Y,Z)=g\bigl(Y,(\widehat{\nabla}_X\varphi)(Z)\bigr)$
for all $X,Y,Z\in\XM$. Hence
$$
\widehat{\nabla}\Phi=0 \iff \widehat{\nabla}\varphi=0.
$$
Moreover, since $g(X,\xi)=\eps\,\eta(X)$, it follows that $(\widehat{\nabla}_X\eta)(Y)=\eps\,g(Y,\widehat{\nabla}_X\xi)$,
and therefore
$$
\widehat{\nabla}\eta=0\iff\widehat{\nabla}\xi=0.
$$
Consequently, the Levi-Civita connection $\widehat{\nabla}$ is compatible with
$(\varphi,\xi,\eta,g,\eps)$ in the sense of Definition \ref{df::compat} if and
only if
$$
\widehat{\nabla}\varphi=0\qquad\text{and}\qquad\widehat{\nabla}\xi=0.
$$
In particular, on a co-K\"ahler pseudo-metric manifold, $\widehat{\nabla}$ is
always compatible.
\end{uw}

\begin{df}\label{df::acstat}
Let $(M,\varphi,\xi,\eta,g,\eps)$ be an almost contact pseudo-metric manifold
and let $\nabla^s$ be a torsion-free connection on $M$. We say that
$(\nabla^s,\varphi,\xi,\eta,g,\eps)$ is an
\emph{almost contact pseudo-metric statistical structure} if
$\nabla^s$ is statistical with respect to $g$ and compatible with
$(\varphi,\xi,\eta,g,\eps)$ in the sense of Definition \ref{df::compat}.
\end{df}
Standard computations adapted to almost contact structures give the following lemma.
\begin{lm}\label{lm::dual}
Let $(M,\varphi,\xi,\eta,g,\eps)$ be an almost contact pseudo-metric manifold
and let $\nabla^s$ be a connection compatible with
$(\varphi,\xi,\eta,g,\eps)$.
Let $\nabla^d$ be the dual connection of $\nabla^s$ with respect to $g$. Then:
\begin{enumerate}
\item
\begin{equation}\label{eq::dual::formula}
\nabla^d_XY=-\varphi\nabla^s_X\varphi Y+X(\eta(Y))\xi=\nabla^s_XY-\varphi\bigl((\nabla^s_X\varphi)(Y)\bigr)
\end{equation}
for all $X,Y\in\XM$. In particular,
\begin{equation}\label{eq::dual::phi}
\varphi\nabla^d_XY=\nabla^s_X\varphi Y,
\qquad
\varphi\nabla^s_XY=\nabla^d_X\varphi Y,
\qquad
\nabla^s_X\xi=\nabla^d_X\xi=0.
\end{equation}
\item
$\nabla^d$ is also compatible with $(\varphi,\xi,\eta,g,\eps)$.
\item
If $T^s$ and $T^d$ denote the torsion tensors of $\nabla^s$ and $\nabla^d$,
respectively, then
\begin{equation}\label{eq::dual::torsion}
T^d(X,Y)-T^s(X,Y)=\varphi\Bigl((\nabla^s_Y\varphi)(X)-(\nabla^s_X\varphi)(Y)\Bigr)
\end{equation}
for all $X,Y\in\XM$.
In particular, if $\nabla^s$ is torsion-free and compatible with
$(\varphi,\xi,\eta,g,\eps)$, then $\nabla^s$ is statistical if and only if
$\nabla^s\varphi$ is symmetric.
\end{enumerate}
\end{lm}

\begin{lm}\label{lm::acstat}
Let $(M,\varphi,\xi,\eta,g,\eps)$ be an almost contact pseudo-metric manifold.
If there exists an almost contact pseudo-metric statistical structure
$(\nabla^s,\varphi,\xi,\eta,g,\eps)$ on $M$, then
$(M,\varphi,\xi,\eta,g,\eps)$ is a co-K\"ahler pseudo-metric manifold.
\end{lm}
\begin{proof}
Let $\nabla^d$ be the dual connection of $\nabla^s$ with respect to $g$.
Since $(\nabla^s,\varphi,\xi,\eta,g,\eps)$ is an almost contact pseudo-metric
statistical structure, $\nabla^s$ is statistical with respect to $g$ and
compatible with $(\varphi,\xi,\eta,g,\eps)$. Hence, by
Lemma \ref{lm::dual}, $\nabla^d$ is also compatible with
$(\varphi,\xi,\eta,g,\eps)$.
Since $\nabla^s$ is statistical, the Levi-Civita connection $\widehat{\nabla}$
of $g$ is given by
$$
\widehat{\nabla}=\frac12(\nabla^s+\nabla^d).
$$
Therefore $\widehat{\nabla}\Phi=0$ and $\widehat{\nabla}\eta=0$.
By Remark \ref{uw::LC::compat}, the equality $\widehat{\nabla}\Phi=0$ implies $\widehat{\nabla}\varphi=0$.
Now Remark \ref{uw::cok::char} implies that  $(M,\varphi,\xi,\eta,g,\eps)$ is a co-K\"ahler pseudo-metric manifold.
\end{proof}

\begin{ex}\label{ex::acstat}
Let $(M,\varphi,\xi,\eta,g,\eps)$ be a co-K\"ahler pseudo-metric manifold and
let $F\in\ker\eta$. Denote $\alpha:=g(F,\cdot)$, $\beta:=g(\varphi F,\cdot)$.
Let $\widehat{\nabla}$ be the Levi-Civita connection of $g$ and define
$$
\nabla^s=\widehat{\nabla}+D^s,
$$
where
\begin{equation}\label{eq::Ds::ac}
D^s_XY=
\bigl(\alpha(X)\beta(Y)+\beta(X)\alpha(Y)\bigr)F
+\bigl(\alpha(X)\alpha(Y)-\beta(X)\beta(Y)\bigr)\varphi F
\end{equation}
for all $X,Y\in\XM$.
\par Then $(\nabla^s,\varphi,\xi,\eta,g,\eps)$ is an almost contact pseudo-metric
statistical structure. Moreover, its dual connection is $\nabla^d=\widehat{\nabla}-D^s$.
\end{ex}

\begin{ex}\label{ex::independent}
The conditions of being statistical and compatible with the almost contact
pseudo-metric structure are independent.
\par Let $(M,\varphi,\xi,\eta,g,\eps)$ be a co-K\"ahler pseudo-metric manifold and
let $F\in\ker\eta$ be a nonzero vector field. Denote
$\alpha:=g(F,\cdot)$, $\beta:=g(\varphi F,\cdot)$,
and let $\widehat{\nabla}$ be the Levi-Civita connection of $g$.
For each of the following choices of $D^s$, put $\nabla^s=\widehat{\nabla}+D^s$.
\begin{enumerate}
\item
If $D^s_XY=\alpha(X)\alpha(Y)\,\varphi F$, then $\nabla^s$ is compatible with $(\varphi,\xi,\eta,g,\eps)$, but it is not statistical.
\item
If $D^s_XY=\alpha(X)\alpha(Y)F$, then $\nabla^s$ is statistical, but it is not compatible with $(\varphi,\xi,\eta,g,\eps)$.
\item
If $D^s_XY=\bigl(\alpha(X)\beta(Y)+\beta(X)\alpha(Y)\bigr)F+\alpha(X)\alpha(Y)\,\varphi F$,
then $\nabla^s$ is again statistical, but it is not compatible with $(\varphi,\xi,\eta,g,\eps)$.
\end{enumerate}
\end{ex}

\begin{ex}\label{ex::eta::needed}
The condition $\nabla^s\eta=0$ does not follow from $\nabla^s\Phi=0$, even if $\nabla^s$ is statistical.
\par Let $(M,\varphi,\xi,\eta,g,\eps)$ be a co-K\"ahler pseudo-metric manifold and
let $\widehat{\nabla}$ be the Levi-Civita connection of $g$. Define
$$
\nabla^s=\widehat{\nabla}+D,\qquad D_XY:=\eta(X)\eta(Y)\,\xi
$$
for all $X,Y\in\XM$.
\par Then $\nabla^s$ is torsion-free, since $D$ is symmetric. Moreover,
$$
g(D_XY,Z)=\eps\,\eta(X)\eta(Y)\eta(Z),
$$
which is totally symmetric in $X,Y,Z$. Hence $\nabla^s$ is statistical with
respect to $g$.
\par Since $(M,\varphi,\xi,\eta,g,\eps)$ is co-K\"ahler, we have
$\widehat{\nabla}\Phi=0$ ,$\widehat{\nabla}\eta=0$.
Furthermore, $D_XY$ is always proportional to $\xi$, and therefore
$$
\Phi(D_XY,Z)=\Phi(Y,D_XZ)=0
$$
for all $X,Y,Z\in\XM$. Therefore, $(\nabla^s_X\Phi)(Y,Z)=0$, that is,
$$
\nabla^s\Phi=0.
$$
On the other hand,
$$
(\nabla^s_X\eta)(Y)=(\widehat{\nabla}_X\eta)(Y)-\eta(D_XY)=-\eta(X)\eta(Y),
$$
so $\nabla^s\eta\neq 0$.
\end{ex}

\begin{lm}\label{lm::dalpha}
Let $(M,\varphi,\xi,\eta,g,\eps)$ be an almost contact pseudo-metric manifold
and let $\nabla^s$ be a statistical connection with dual connection $\nabla^d$
with respect to $g$. Let $F\in\XM$ and define $\alpha:=g(F,\cdot)$, $\beta:=g(\varphi F,\cdot)$.
Then
\begin{equation}\label{eq::dalpha}
d\alpha(U,V)=g(\nabla^s_UF,V)-g(U,\nabla^s_VF)=g(\nabla^d_UF,V)-g(U,\nabla^d_VF)
\end{equation}
for all $U,V\in\XM$.
Moreover,
\begin{equation}\label{eq::dbeta}
d\beta(U,V)=g(\nabla^s_U\varphi F,V)-g(U,\nabla^s_V\varphi F)=g(\nabla^d_U\varphi F,V)-g(U,\nabla^d_V\varphi F)
\end{equation}
for all $U,V\in\XM$.
If, in addition, $\nabla^s$ is compatible with $(\varphi,\xi,\eta,g,\eps)$, then
\begin{equation}\label{eq::dbeta2}
d\beta(U,V)=g(\varphi\nabla^d_UF,V)-g(U,\varphi\nabla^d_VF)=g(\varphi\nabla^s_UF,V)-g(U,\varphi\nabla^s_VF).
\end{equation}
\end{lm}

\begin{proof}
Since $\nabla^d$ has no torsion,
\begin{align*}
d\alpha(U,V)
&=U\alpha(V)-V\alpha(U)-\alpha([U,V])\\
&=g(\nabla^s_UF,V)+g(F,\nabla^d_UV)
-g(\nabla^s_VF,U)-g(F,\nabla^d_VU)-g(F,[U,V])\\
&=g(\nabla^s_UF,V)-g(U,\nabla^s_VF),
\end{align*}
which proves the first equality in \eqref{eq::dalpha}. Analogously, using the
fact that $\nabla^s$ is torsion-free, we obtain the second equality in \eqref{eq::dalpha}.
Applying the same argument to the vector field $\varphi F$ we obtain \eqref{eq::dbeta}.
\par Assume now that $\nabla^s$ is compatible with
$(\varphi,\xi,\eta,g,\eps)$. Then, by Lemma \ref{lm::dual},
$\nabla^s_U\varphi F=\varphi\nabla^d_UF$ and $\nabla^d_U\varphi F=\varphi\nabla^s_UF$.
Substituting these equalities into \eqref{eq::dbeta}, we obtain \eqref{eq::dbeta2}.
\end{proof}
\begin{df}
We say that a connection $\nabla$ on the product of manifolds $M_1\times M_2$ is a weak product connection if
\[
\nabla_XY=0
\]
for $X,Y$ being pull-back vector fields from different product factors  and if additionally for $X,\widehat X\in\Gamma T_i$ 
\[
\nabla_X\widehat X\in\Gamma T_i
\]
for any $i\in\{1,2\}$, where $TM=T_1\oplus T_2$ is the induced splitting of the tangent bundle.
\end{df}

\begin{tw}\label{tw::acstat}
Let $(M,\varphi,\xi,\eta,g,\eps)$ be an almost contact pseudo-metric manifold
and let $(\nabla^s,\varphi,\xi,\eta,g,\eps)$ be an almost contact pseudo-metric
statistical structure. Let $F\in\ker\eta$ be a vector field such that
$g(F_p,F_p)\neq 0$, $g(F_p,F_p)\neq -1$ for all $p\in M$.
Then the following conditions are equivalent:
\begin{enumerate}
\item $\nabla^s$ is compatible with $(\varphi,\xi,\eta,g^\ast,\eps)$.
\item $\nabla^s$ is statistical with respect to $g^\ast$.
\item $(M,\varphi,\xi,\eta,g,\eps)$ is locally isometric and $\varphi$-holomorphic to
$$
\bigl(N\times \R^2,\widetilde\varphi,\widetilde\xi,\widetilde\eta,g|_N+\psi(t,s)(dt^2+ds^2),\eps\bigr),
$$
where $N$ is a leaf of the distribution $H^{\perp_g}:=\Span\{F,\varphi F\}^{\perp_g}$,
$\psi$ is a nowhere vanishing function, and $(\widetilde\varphi,\widetilde\xi,\widetilde\eta)$ is the
almost contact structure on $N\times\R^2$ given by
\begin{align*}
\widetilde\varphi(X+a\partial_t+b\partial_s)=\varphi_NX+a\partial_s-b\partial_t, \\
\widetilde\xi=\xi_N,\quad \widetilde\eta(X+a\partial_t+b\partial_s)=\eta_N(X)
\end{align*}
for every $X\in\X(N)$ and $a,b\in\Cniesk(N\times\R^2)$.
\end{enumerate}
Moreover, in this case $\nabla^s$ is a weak product connection and $g(F,F)$ is constant.
\end{tw}
\begin{proof}
By Lemma \ref{lm::acstat}, the existence of the almost contact pseudo-metric
statistical structure $(\nabla^s,\varphi,\xi,\eta,g,\eps)$ implies that
$(M,\varphi,\xi,\eta,g,\eps)$ is a co-K\"ahler pseudo-metric manifold.

\medskip

\noindent\emph{$(1)\Rightarrow(2)$.}
Let $\nabla^d$ denote the dual connection of $\nabla^s$ with respect to $g$.
Since $\nabla^s$ is statistical with respect to $g$, the connection $\nabla^d$
is torsion-free. Let $\nabla^{d,\ast}$ denote the dual connection of $\nabla^s$
with respect to $g^\ast$. Since $\nabla^s$ is compatible with
$(\varphi,\xi,\eta,g^\ast,\eps)$, Lemma \ref{lm::dual} gives
$$
\nabla^{d,\ast}_XY=-\varphi\nabla^s_X\varphi Y+X(\eta(Y))\,\xi
$$
for all $X,Y\in\XM$. On the other hand, because $\nabla^s$ is compatible with
$(\varphi,\xi,\eta,g,\eps)$, the same lemma gives
$$
\nabla^d_XY=-\varphi\nabla^s_X\varphi Y+X(\eta(Y))\,\xi.
$$
Hence $\nabla^{d,\ast}=\nabla^d$.
Since $\nabla^d$ has no torsion, $\nabla^{d,\ast}$ has no torsion as well.
Therefore $\nabla^s$ is statistical with respect to $g^\ast$.

\medskip

\noindent\emph{$(2)\Rightarrow(3)$.}
Let us denote
$\alpha:=g(F,\cdot)$, $\beta:=g(\varphi F,\cdot)$, $\mathcal D:=H^{\perp_g}=\Span\{F,\varphi F\}^{\perp_g}$.
Since $F\in\ker\eta$ and $g(\varphi F,\varphi F)=g(F,F)$, $g(F,\varphi F)=0$,
the distribution $\mathcal D$ is well defined and
$$
TM=\mathcal D\oplus \Span\{F,\varphi F\}.
$$
Define
\begin{equation}\label{eq::Delta}
\Delta(U,V,W):=(\nabla^s_Ug^\ast)(V,W)-(\nabla^s_Vg^\ast)(U,W)
\end{equation}
for all $U,V,W\in\XM$.
Using $g^\ast=g+\alpha\otimes\alpha+\beta\otimes\beta$
and the fact that $\nabla^sg$ is totally symmetric, we obtain
\begin{align*}
\Delta(U,V,W)
&=(\nabla^s_U(\alpha\otimes\alpha))(V,W)-(\nabla^s_V(\alpha\otimes\alpha))(U,W)\\
&\quad+
(\nabla^s_U(\beta\otimes\beta))(V,W)-(\nabla^s_V(\beta\otimes\beta))(U,W).
\end{align*}
Expanding and using Lemma \ref{lm::dalpha}, we get
\begin{align}
\Delta(U,V,W)
=d\alpha(U,V)\alpha(W)+d\beta(U,V)\beta(W) +
g\bigl(\alpha(V)\nabla^d_UF-\alpha(U)\nabla^d_VF,W\bigr)\notag\\
+g\bigl(\beta(V)\nabla^d_U\varphi F-\beta(U)\nabla^d_V\varphi F,W\bigr).\notag
\end{align}
Evaluating above on the decomposition
$$
TM=\mathcal D\oplus \Span\{F,\varphi F\}
$$
for $X,Y,Z\in\Gamma(\mathcal D)$
we obtain
\begin{align}
\Delta(X,Y,F)
&=g(F,F)\,d\alpha(X,Y), \label{eq::DeltaXYF}\\
\Delta(X,Y,\varphi F)
&=g(F,F)\,d\beta(X,Y), \label{eq::DeltaXYphiF}\\
\Delta(X,F,Z)
&=g(F,F)\,g(\nabla^d_XF,Z), \label{eq::DeltaXFZ}\\
\Delta(X,\varphi F,Z)
&=g(F,F)\,g(\nabla^d_X\varphi F,Z), \label{eq::DeltaXphiFZ}\\
\Delta(F,\varphi F,Z)
&=g(F,F)\,g([F,\varphi F],Z), \label{eq::DeltaFphiFZ}\\
\Delta(F,\varphi F,\varphi F)
&=g(F,F)\,F\bigl(g(F,F)\bigr), \label{eq::DeltaFphiFphiF}\\
\Delta(F,\varphi F,F)
&=-g(F,F)\,\varphi F\bigl(g(F,F)\bigr), \label{eq::DeltaFphiFF}\\
\Delta(X,\varphi F,\varphi F)
&=g(F,F)\Bigl(
X\bigl(g(F,F)\bigr)-g(X,\nabla^s_{\varphi F}\varphi F)
\Bigr). \label{eq::DeltaXphiFphiF}
\end{align}
Since $\Delta=0$ ($\nabla^s$ is statistical with respect to $g^\ast$) and $g(F,F)\neq 0$,
it follows from \eqref{eq::DeltaXYF} and \eqref{eq::DeltaXYphiF} that
$$
d\alpha(X,Y)=0,\qquad d\beta(X,Y)=0
$$
for all $X,Y\in\Gamma(\mathcal D)$.
Hence, by the Fr\"obenius theorem, the distribution
$$
\mathcal D=\ker\alpha\cap\ker\beta
$$
is integrable. The formula \eqref{eq::DeltaFphiFZ} implies that
distribution $\Span\{F,\varphi F\}$ is also integrable.
\par Moreover, \eqref{eq::DeltaXFZ} and \eqref{eq::DeltaXphiFZ} imply
$$
\nabla^d_XF\in\Gamma(\Span\{F,\varphi F\}),\qquad \nabla^d_X\varphi F\in\Gamma(\Span\{F,\varphi F\})
$$
for every $X\in\Gamma(\mathcal D)$.
Now (thanks to Lemma \ref{lm::dual}) we obtain
$$
\nabla^s_XF,\ \nabla^s_X\varphi F,\ \nabla^d_XF,\ \nabla^d_X\varphi F
\in \Gamma(\Span\{F,\varphi F\})
$$
for every $X\in\Gamma(\mathcal D)$.
\par Let $X,Y\in\Gamma(\mathcal D)$ and $U\in\Gamma(\Span\{F,\varphi F\})$.
Since $g(Y,U)=0$, we have
$$
0=Xg(Y,U)=g(\nabla^s_XY,U)+g(Y,\nabla^d_XU).
$$
The second term vanishes because $\nabla^d_XU\in\Gamma(\Span\{F,\varphi F\})$.
Hence $g(\nabla^s_XY,U)=0$ for every $U\in\Gamma(\Span\{F,\varphi F\})$, and therefore
$$
\nabla^s_XY\in\Gamma(\mathcal D), \ \ \ \nabla^d_XY\in\Gamma(\mathcal D).
$$

A direct computation together with the symmetry of
$\nabla^s\varphi$, shows that
$
g(X,\nabla^d_UV)=0
$
for every $X\in\Gamma(\mathcal D)$ and all
$U,V\in\Gamma(\Span\{F,\varphi F\})$.
Hence
$$
\nabla^d_UV\in\Gamma(\Span\{F,\varphi F\}).
$$
Using again Lemma \ref{lm::dual}, we obtain the same conclusion for $\nabla^s$.
Therefore the distributions
$\mathcal D$ and $\Span\{F,\varphi F\}$
are both $\nabla^s$- and $\nabla^d$-parallel.
\par Since $\widehat{\nabla}=\frac{1}{2}(\nabla^s+\nabla^d)$,
both distributions are also $\widehat{\nabla}$-parallel. As they are
complementary, orthogonal and nondegenerate, the local de Rham decomposition
theorem, see \cite{M}, yields a local isometric decomposition
$$
(M,g)\cong (N\times S,g_1+g_2),
$$
where $TN=\mathcal D$ and $TS=\Span\{F,\varphi F\}$.
Because $\mathcal D$ and $TS$ are $\varphi$-invariant, the almost contact
structure decomposes as
$$
(\varphi,\xi,\eta)\cong (\varphi_N,\xi_N,\eta_N)\oplus J_2,
$$
where $J_2$ is the induced complex structure on the two-dimensional manifold $S$.
We may choose local coordinates $(t,s)$ on $S$ such that $(S,J_2)\cong (\R^2,J_0)$, where
$J_0(\partial_t)=\partial_s$, $J_0(\partial_s)=-\partial_t$.
Thus $(M,\varphi,\xi,\eta,g,\eps)$ is locally isometric and $\varphi$-holomorphic
to
$$
\bigl(N\times \R^2,\widetilde\varphi,\widetilde\xi,\widetilde\eta,g_1+g_2,\eps\bigr),
$$
with
$$
\widetilde\varphi(X+a\partial_t+b\partial_s)=\varphi_NX+a\partial_s-b\partial_t,
$$
$$
\widetilde\xi=\xi_N,\qquad \widetilde\eta(X+a\partial_t+b\partial_s)=\eta_N(X).
$$
Since $g_2$ is compatible with $J_0$, there exists a nowhere vanishing function
$\psi$ such that $g_2=\psi(t,s)(dt^2+ds^2)$.

\par

 Furthermore,  the connection $\nabla^s$ is a weak product
connection: if $X,Y$ are pull-back vector fields  from different
factors, then
\[
\widehat{\nabla}_XY=0
\]
since we already know that the manifold is a product.
Next, using the formulas from Lemma \ref{lm::dual}, we obtain
\[
2\nabla^s_XY=\varphi\bigl((\nabla^s_X\varphi)(Y)\bigr).
\]

In particular, using that the two distributions are $\nabla^s$-parallel, $(\nabla^s_X\varphi)(Y)$ is a section of $H$ if $Y$ is a section of
$H$ and $X$ is a section of $H^\perp$, and it is a section of $H^\perp$
in the opposite case.
Now, using the symmetry of $\nabla\varphi$ and once again the fact that both
distributions are parallel, we obtain that the right-hand side belongs
simultaneously to $H$ and $H^\perp$. Hence it vanishes.

\par It remains to prove that $g(F,F)$ is constant. Since $\Delta=0$ and
$g(F,F)\neq 0$, from \eqref{eq::DeltaFphiFphiF} and \eqref{eq::DeltaFphiFF}
we immediately obtain
$$
F\bigl(g(F,F)\bigr)=0,\qquad\varphi F\bigl(g(F,F)\bigr)=0.
$$
From \eqref{eq::DeltaXphiFphiF}, we also get
$$
X\bigl(g(F,F)\bigr)-g(X,\nabla^s_{\varphi F}\varphi F)=0
$$
for $X\in\Gamma(\mathcal D)$.
Since $\Span\{F,\varphi F\}$ is $\nabla^s$-parallel, we have
$\nabla^s_{\varphi F}\varphi F\in\Gamma(\Span\{F,\varphi F\})$
thus $g(X,\nabla^s_{\varphi F}\varphi F)=0$.
Therefore $X\bigl(g(F,F)\bigr)=0$ and in consequence $g(F,F)$ is constant.

\medskip

\noindent\emph{$(3)\Rightarrow(1)$.}
Suppose now that $(M,\varphi,\xi,\eta,g,\eps)$ is locally isometric and
$\varphi$-holomorphic to
$$
\bigl(N\times \R^2,\widetilde\varphi,\widetilde\xi,\widetilde\eta,
g|_N+\psi(t,s)(dt^2+ds^2),\eps\bigr),
$$
and that $\nabla^s$ is a weak product connection. Since $g(F,F)=\lambda$ is
constant we have
$$
g^\ast=g|_N+(1+\lambda)\psi(t,s)(dt^2+ds^2)
$$
and
$$
\Phi^\ast=\Phi_N+(1+\lambda)\Omega_2,
$$
where $\Phi_N$ is the fundamental form of the first factor and $\Omega_2$ is
the standard fundamental form of the second one. Since $\nabla^s$ is a weak  product
connection, compatible with $(\varphi,\xi,\eta,g,\eps)$ and $1+\lambda$ is
constant, it follows immediately that $\nabla^s\Phi^\ast=0$. Of course $\nabla^s\eta=0$ thus
$\nabla^s$ is compatible with $(\varphi,\xi,\eta,g^\ast,\eps)$.
This completes the proof.
\end{proof}
\begin{uw}
    An analogous to the above theorem in holomorphic statistical setting have been obtained in \cite{MNO}, Theorem 4.8. However,  their  analogous condition in (3) claims that $\nabla^s$ is a product connection in the standard sense. This is not correct as taking $M=\mathbb{R}^4=\mathbb{C}\oplus\mathbb{C}$ with the induced complex structure, with the pseudo-K\"ahler metric $g=dx^2+dy^2-dz^2-dw^2$ and $F=\sqrt{2}\partial_z$, we have that the connection $\nabla^s$ given by the only non-zero Christoffel symbols $$ \Gamma^1_{12}=\Gamma^1_{21}=z,\ \Gamma^2_{11}=z,\ \Gamma^2_{22}=-z  $$ 
    is statistical with respect to both $g$ and its Hermitian variation but is not a product connection (but it is a weak product connection). By taking an analogous structure in our case (by taking trivial almost contact metric structure above $M$) we obtain that also in our case $\nabla^s$ may not be a product connection. 
\end{uw}
\section*{Acknowledgments}
The first author has been supported by the Polish National Science Center, project number 2022/47/D/ST1/02197. This research was financed by the Ministry of Science and Higher Education of the Republic of Poland.

\end{document}